\documentclass[reqno,11pt]{amsart}
\usepackage{amssymb, graphicx}
\usepackage[title,titletoc,header]{appendix}

\newtheorem{lemma}{Lemma}[section]
\newtheorem{theorem}{Theorem}[section]
\newtheorem{proposition}{Proposition}[section]

\theoremstyle{definition}
\newtheorem{definition}{Definition}[section]
\theoremstyle{remark}
\newtheorem{remark}{Remark}[section]

\numberwithin{equation}{section}

\newcommand{\p}{\partial}
\newcommand{\norm}[1]{\left\Vert#1\right\Vert}
\newcommand{\dd}{\mathrm{d}}
\newcommand{\di}{\mathrm{div}}
\newcommand{\grad}{\mathrm{grad}}
\newcommand \R{\mathbb{R}}
\newcommand \T{\mathbb{T}}
\newcommand \Z{\mathbb{Z}}
\makeatletter
\newcommand{\rmnum}[1]{\mathrm{\romannumeral #1}}
\newcommand{\Rmnum}[1]{\mathrm{\expandafter\@slowromancap\romannumeral#1@}}
\makeatother

\begin{document}
\title[Subsonic flow passing a duct with friction]
{Subsonic flows passing a duct for three-dimensional steady compressible Euler system with friction}

\author{Hairong Yuan}
\address{Hairong Yuan:
Department of Mathematics, Center for Partial Differential Equations, and Shanghai Key Laboratory of Pure Mathematics and Mathematical Practice,
East China Normal University, Shanghai 200241, China}
\email{hryuan@math.ecnu.edu.cn}

\author{Qin Zhao}
\address{Qin Zhao (Corresponding author): School of Mathematical Sciences, Shanghai Jiao Tong University,
	Shanghai 200240, China}
\email{zhao@sjtu.edu.cn}



\begin{abstract}
For the three-dimensional steady non-isentropic compressible Euler system with friction,  we show existence of a class of symmetric subsonic, supersonic and transonic-shock solutions in a straight duct with constant square-section. Such flows are called Fanno flow in engineering. We formulate a boundary value problem for subsonic flows, and study their stability under multidimensional small perturbations of boundary conditions. Since the subsonic Euler system is of elliptic-hyperbolic composite-mixed type, this is achieved by using the framework established in [L. Liu; G. Xu; H. Yuan: Stability of spherically symmetric subsonic flows and
transonic shocks under multidimensional perturbations. Adv. Math. 291 (2016), 696--757], and establishing an iteration scheme, which involves solving a second order nonlocal elliptic equation.
\end{abstract}
\maketitle

\section{Introduction}\label{sec1}

We consider polytropic gases, namely perfect fluids with the constitutive relation $p=A(s)\rho^\gamma$, that move steadily in a three-dimensional duct.
Here $p, \rho$, and $s$ represent the pressure, density of mass, and entropy of the flow respectively, and $\gamma>1$ is the adiabatic exponent, while $A(s)=k_0\exp(s/c_\nu)$ for two positive constants $k_0$ and $c_\nu$. In the Descartesian coordinates $(x^0, x^1,x^2)$ of the Euclidean space $\R^3$, let $D=\{(x^0,x^1,x^2): x^0\in(0,L), (x^1,x^2)\in(0,\pi)\times(0,\pi)\}$ be a rectilinear duct with length $L$ and constant square cross-section. Suppose the gas flows in the duct bearing frictions due to non-smoothness of the lateral walls.
It is noted that in gas dynamics, such compressible frictional flow is called as {\it Fanno flow} \cite[Section 2.16]{Far2008}. It is adopted frequently in engineering, including transport of natural gas in long pipe lines, the design and analysis of nozzles, etc. We wish to study such flows rigorously from mathematical point of view and find the roles played by friction, by considering the three-dimensional steady non-isentropic compressible Euler system.

Let $u=(u^0, u^1, u^2)^\top$ be the velocity of the gases. Recall that the flow is {\it subsonic} if $|u|<c$, and {\it supersonic} if $|u|>c$, where $c=\sqrt{\gamma p/\rho}$ is the sonic speed. It turns out that friction has different effects for subsonic flow and supersonic flow, and it also supports transonic shocks in ducts. In this paper, we firstly show existence of a class of subsonic flow and supersonic flow, as well as transonic shocks that depend only on $x^0 $ in $D$, analyze the effects of friction, and then study general perturbed subsonic flows. These results  will be used to study effects of frictions for transonic shocks in another paper.

To avoid technical difficulties arose by the lateral walls, as in \cite{CY,xiechen}, by assuming the flows have some symmetric properties with respect to the walls $[0,L]\times\p[0,\pi]^2$, we may suppose the flows are periodic in $x^1, x^2$-directions with periods $2\pi$. It usually happens that the gas moves mainly along the $x^0$-direction, that means, $u^0>0$ and $u^1, u^2$ are small comparing to $u^0$. Hence without loss of generality, we may assume that the friction force is acting on the negative $x^0$-direction, and equals $\mu (u^0)^2$ per unit mass of the gas. Here $\mu$ is a positive constant that may depend on the Mach number $M=|u|/c$ of the flow. Following the practice in gas dynamics \cite[Section 2.16]{Far2008}, we assume it is a constant. Then the motion of such flow with friction is governed by the following Euler system with a damping term ({\it cf.} \cite{CF, Da, Sha1953}):
\begin{eqnarray}
\di (\rho u\otimes u)+\grad\, p-\rho \mathfrak{b}&=&0,\label{eq101}\\
\di (\rho u)&=&0,\label{eq102}\\
\di (\rho E u)-\rho \mathfrak{b}\cdot u&=&0,\label{eq103}
\end{eqnarray}
where `$\di$' and `$\grad$' are respectively the standard divergence and gradient operator, $E\triangleq\frac{1}{2}|u|^2+\frac{\gamma}{\gamma-1}\frac{p}{\rho}$  is the so-called  {\it Bernoulli constant},  and $\mathfrak{b}=(-\mu (u^0)^2,0,0)^\top$.  These equations are the conservation of momentum, mass and energy, respectively.

The main difficulty of studying subsonic flows using the stationary compressible Euler system comes from the fact that it is of elliptic-hyperbolic composite-mixed type \cite{CY,LXY2016}, and there is no general theory to treat such equations. So it is even challenging to formulate a well-posed boundary value problem for studying subsonic Euler flows ({\it cf.} \cite{Yu1}), which is both important for theoretical analysis and numerical computations. It is more difficult to study the three-dimensional Euler system, since one cannot introduce Lagrange-type transform as in two-dimensional case \cite{cdx,Yu1}. In \cite{LXY2016}, the authors proposed a general framework to decompose the three-dimensional steady Euler system and  linearize it at a special symmetric solution. So in this paper we will mainly utilize the ideas and results established in \cite{LXY2016}. Comparing to the geometric effects considered in \cite{LXY2016}, appearance of friction leads to variation of  Bernoulli constant $E$ along flow trajectories, hence we were led to a nonlocal elliptic equation involving an integral term even for purely subsonic flows.

We also remark that the works of Chen and Xie \cite{xiechen} considering three-dimensional subsonic isentropic flows without frictions, based on a totally different method from ours.  One may consult  \cite{xiechen,LXY2016,weng} and references listed there for some results on subsonic flows and transonic shocks in Euler flows.  There are also many significant works on Fanno flows by establishing global weak solutions using generalized Glimm scheme or Gudonov scheme to the one-dimensional unsteady compressible Euler system with frictions, see \cite{CHHQ2016,Ts2015}. Huang, Pan and Wang et. al. also studied  weak solutions of the one dimensional unsteady compressible Euler system with a damping term (that is $\mathfrak{b}=-\mu u^0$), see for example, \cite{HPW} and references therein.  These studies, from different points of view, definitely help us understand better the role of friction in gas dynamics. See also \cite{BDX} for a study on the effect of electric field on subsonic flows.

The rest of the paper is organized as follows. In Section \ref{sec2} we show existence of special subsonic, supersonic and transonic shock solutions by studying some ordinary differential equations. Then we formulate a boundary value problem (S) to study general perturbed subsonic flows, and state the main result of this paper, namely Theorem \ref{thm21}. In Section \ref{sec3} we reformulate problem (S) to a more tractable problem (S3), by using a decomposition of the system \eqref{eq101}-\eqref{eq103} established in \cite{LXY2016}. In Section \ref{sec4}, we study a crucial mixed boundary value problem of a second order nonlocal elliptic equation. In Section \ref{sec5}, by showing contraction of a nonlinear mapping, we prove Theorem \ref{thm21}.
We remark that at first glance, many computations and expressions of this paper are quite similar to that of \cite{LXY2016}, however, we need to be very careful to carry out all the detailed calculations, and there are some subtle differences, since its the single term $\mathfrak{b}$ that drastically changes the overall behaviour of solutions of the Euler system, and the major differences lie in these details.

\section{Special solutions and main result }\label{sec2}

Let $\Omega=\{(x^0, x^1,x^2): x^0\in(0, L),x'=(x^1,x^2)\in\T^2\}$ be the duct we consider henceforth, where $\T^2=\{(x^1,x^2): x^1,x^2\in[0,2\pi]\}$ is the flat $2$-torus. Then $\p\Omega$,  the boundary of $\Omega$, is given by $\Sigma_0\cup \Sigma_1$, with $\Sigma_0=\{0\}\times \T^2$ and $\Sigma_1=\{L\}\times \T^2$. For the velocity vector $u=(u^0, u^1,u^2)^\top$, for convenience, in this paper we call $u^0$ the {\it normal velocity} and $u'=(u^1,u^2)^\top$ the {\it tangential velocity}.

\subsection{Special solutions}
Suppose that the flow depends only on $x^0$,  the normal velocity $u^0$ is positive, and the tangential velocity $u'$ is identically zero in $\Omega$. Then the system \eqref{eq101}-\eqref{eq103} is reduced to the following ordinary differential equations:
\begin{equation}\label{ode-1}
\begin{cases}
\frac{\dd }{\dd x^0}(\rho u)=0,\\
\frac{\dd }{\dd x^0}(\rho u^2+p)=-\mu \rho u^2,\\
\frac{\dd }{\dd x^0}\Big(\rho (\frac{1}{2}u^2+\frac{\gamma p}{(\gamma-1)\rho})u\Big)=-\mu \rho u^3.
\end{cases}
\end{equation}
We could solve for continuous flow with Mach number $M\ne1$ that
\begin{eqnarray}
\frac{\dd u}{\dd x^0}&=&\frac{\mu u M^2}{1-M^2},\label{eq25}\\
\frac{\dd \rho}{\dd x^0}&=&\frac{\mu\rho M^2}{M^2-1},\label{eq26}\\
\frac{\dd p}{\dd x^0}&=&\frac{\mu\gamma p M^2}{M^2-1}.\label{eq27}
\end{eqnarray}
Here, for simplicity, instead of $u^0(x^0)$, we have written  $u=u(x^0)$ to be the normal velocity.
By the fact that $E=\frac12 u^2+\frac{\gamma p}{(\gamma-1)\rho}$ and $p=A(s)\rho^\gamma$, it follows that
\begin{eqnarray}
\frac{\dd E}{\dd x^0}&=&-\mu u^2=-\frac{2\mu(\gamma-1)M^2}{(\gamma-1)M^2+2}E,\label{eq28E}\\
\frac{\dd s}{\dd x^0}&=&0,\label{eq29S}\\
\frac{\dd\theta}{\dd x^0}&=& (\gamma-1)\frac{\mu\theta M^2}{M^2-1},
\end{eqnarray}
where $\theta=p/\rho R$ is the temperature of the gas, and $R$ is a universal positive constant.
It is important to observe that the equation satisfied by the Mach number is decoupled:
\begin{eqnarray}\label{eqmach}
\frac{\dd M}{\dd x^0}&=&\frac{\mu(\gamma+1)M^3}{2(1-M^2)}.
\end{eqnarray}

\subsubsection{Subsonic flows}\label{sec211}
Suppose that the flow is subsonic in the duct, namely $M<1$ for $x^0\in[0,L]$. We easily see that
\[
\frac{\dd u}{\dd x^0}>0,\quad
\frac{\dd \rho}{\dd x^0}<0,\quad
\frac{\dd p}{\dd x^0}<0,\quad
\frac{\dd M}{\dd x^0}>0,\quad \frac{\dd \theta}{\dd x^0}<0;
\]
that is, the density, pressure and temperature decrease while the velocity and Mach number increase along the flow direction for subsonic Fanno flow. Furthermore, if the flow is continuous, then by uniqueness of solutions of Cauchy problem of ordinary differential equations,    $M(0)<1$ implies the flow is always subsonic in the duct.

Let the Mach number of the flow at the entry $\{x^0=0\}$ and the exit $\{x^0=L\}$ be $M_0$ and $M_1$, respectively. Suppose that $0<M_0<M_1<1$.  Then
integrating \eqref{eqmach} for $x^0$ from $0$ to $l$ yields
\[
\int_{M_0}^{M_1}\frac{1-M^2}{M^3}\dd M=\int_{0}^{l}\frac{\mu(\gamma+1)}{2}\dd x^0,
\]
and we obtain that
\begin{eqnarray}\label{eqL}
l=\frac{\frac{1}{M_0^2}-\frac{1}{M_1^2}+\ln M_0^2-\ln M_1^2}{\mu(\gamma+1)}.
\end{eqnarray}
Therefore, the maximal length of a duct for a subsonic flow accelerating from $M_0<1$ to the sonic speed ($M_1=1$) is
\begin{eqnarray}\label{eq29new}
L_{M_0}=\frac{\frac{1}{M_0^2}+\ln M_0^2-1}{\mu(\gamma+1)}>0.
\end{eqnarray}
So for a duct which is longer than $L_{M_0}$, chocking phenomena shall occur and such steady subsonic flow pattern is impossible. From \eqref{eqL}, we also solve $M(x^0)\in(0,1)$ for $x^0\in(0,L_{M_0})$ by
\begin{eqnarray}
\frac{1}{M(x^0)^2}+\ln M(x^0)^2=\frac{1}{M_0^2}+\ln M_0^2-\mu(\gamma+1)x^0.
\end{eqnarray}
Hence we may solve $p(x^0), \rho(x^0)$, $u(x^0)$, $E(x^0)$, $s(x^0)$ and $\theta(x^0)$ once their values on the entry or exit are given.

Therefore we constructed a family of special subsonic solutions to the Euler system \eqref{eq101}-\eqref{eq103}, and it is important to notice that they depend analytically on the parameters $\{\gamma>1, \mu>0, M_0\in(0,1), L\in(0, L_{M_0}), p(L)>0, s(0)>0\}$.

\subsubsection{Supersonic flows}
Suppose now  the flow is supersonic in the duct (i.e. $M>1$ for $x^0\in[0,L]$). We see that
\[
\frac{\dd u}{\dd x^0}<0,\quad
\frac{\dd \rho}{\dd x^0}>0,\quad
\frac{\dd p}{\dd x^0}>0,\quad
\frac{\dd M}{\dd x^0}<0,\quad \frac{\dd \theta}{\dd x^0}>0;
\]
that is, the density, pressure and temperature increase while the velocity and Mach number decrease, contrary to subsonic flows. For given $M_0>1$, if the flow is continuous, then it is always supersonic, and the maximal length of the duct is also given by \eqref{eq29new}.

We conclude that there is a family of special supersonic Fanno flows in the duct, and they depend analytically on the parameters $\{\gamma>1, \mu>0, M_0\in(1,\infty), L\in(0, L_{M_0}), p(0)>0, s(0)>0\}$.

\subsubsection{Transonic shock}
For this phenomena the flow is supersonic near the entry, with Mach number $M_0>1$ at $x^0=0$, and  a shock front intervenes in the duct, which ended the supersonic flow, and the flow behind of it is subsonic. Let $U_+$ (respectively $U_-$) be the state of the flow on the side behind (respectively, ahead) of the shock front. Then the following Rankine-Hugoniot jump conditions should hold:
\begin{eqnarray*}
\rho_-u_-&=&\rho_+u_+,\\
\rho_-u_-^2+p_-&=&\rho_+u_+^2+p_+,\\
\left(\frac{1}{2}u_-^2+\frac{\gamma p_-}{(\gamma-1)\rho_-}\right)\rho_-u_-&=& \left(\frac{1}{2}u_+^2+\frac{\gamma p_+}{(\gamma-1)\rho_+}\right)\rho_+u_+.
\end{eqnarray*}
From these we could solve
\begin{eqnarray}\label{eqmach+}
M_+^2&=&\frac{1+\frac{\gamma-1}{2}M_-^2}{\gamma M_-^2-\frac{\gamma-1}{2}},
\end{eqnarray}
where $M_+$ (respectively $M_-$) is the Mach number behind (respectively, ahead) of the transonic shock front.
The physical entropy condition requires that $M_+<M_-$, so only transonic shock with supersonic flow ($M_->1$) ahead of the shock front, and subsonic flow ($M_+<1$) behind of it is possible.

Let $L_1<L_{M_0}$ be the distance from the entry to the transonic shock front. Utilizing \eqref{eqL}, by
\[
L_1=\frac{\frac{1}{M_0^2}-\frac{1}{M_-^2}+\ln M_0^2-\ln M_-^2}{\mu(\gamma+1)},
\]
we could solve $M_-$. Then by \eqref{eqmach+}, we get $M_+$. So the maximal length from the transonic shock front to the exit is
\[
L_2=L_{M_+}=\frac{\frac{1}{M_+^2}+\ln M_+^2-1}{\mu(\gamma+1)}.
\]
Hence for given $M_0>1$, suppose that $L<L_1+L_2$, we may construct a family of special transonic shock solutions in the duct, and they depend analytically on the parameters $\{\gamma>1, \mu>0, M_0\in(1,\infty), L_1\in(0, L_{M_0}), L\in(0, L_1+L_2), p(L)>0, s(0)>0\}$.

\subsection{Main result}
In this paper we mainly concern existence of general subsonic Fanno flows that are obtained by multidimensional perturbations. To this end, we need to formulate a well-posed boundary value problem for the Euler system \eqref{eq101}-\eqref{eq103}. By some physical considerations (see discussions in \cite[p.706]{LXY2016} and \cite{CF}),
we prescribe the following boundary conditions on $\Sigma_0$:
\begin{eqnarray}\label{eq21}
E=E_0(x'),\quad s=s_0(x'), \quad u'=u'_0(x');
\end{eqnarray}
and the back pressure on  $\Sigma_1$:
\begin{eqnarray}
p=p_1(x'). \label{eq22}
\end{eqnarray}
Here $E_0, s_0$ and $p_1$ are given functions on $\T^2$, while $u'_0(x')$ is a given vector field on $\T^2$.

\medskip
\fbox{
	\parbox{0.90\textwidth}{
		Problem (S): Solve the Euler system \eqref{eq101}-\eqref{eq103} in $\Omega$, subjected to the boundary conditions \eqref{eq21}--\eqref{eq22}.}}

\begin{remark}
For a flow which is obtained by symmetric extension along lateral walls of the duct $D$ and hence periodic in $(x_1, x_2)$, the pressure, density, entropy, normal velocity  shall be even functions with respect to $x^1=k\pi, x^2=k\pi$ for $k\in\mathbb{Z}$, while $u^1$ (respectively $u^2$) shall be odd function along $x^1=k\pi$ (respectively $x^2=k\pi$). So problem (S) is actually a little bit different from the problem of $(x_1,x_2)$-periodic flows reduced by extension of flows in duct, since we neglect these symmetry properties. We do not care about this in-essential technical point here, which can be easily handle in the same way as in \cite{CY}, and we just focus on problem (S) below.
\end{remark}

\medskip

We now pick a special subsonic Fanno flow $U_b=(E_b(x^0), s_b(x^0), u'_b\equiv0, p_b(x^0))$ constructed in Section \ref{sec211}, which is also called a {\it background solution} in the sequel. The following is the main theorem we will prove in this paper.
\begin{theorem} \label{thm21}
Suppose that a background solution $U_b$ satisfies the S-Condition, and $\alpha\in (0,1)$. There exist positive constants $\varepsilon_0$ and $C$ depending only on the background solution $U_b$ and $L, \alpha,\mu, \gamma$ such that if
\begin{multline}\label{eq213}
\norm{E_0(x')-E_b(0)}_{C^{3,\alpha}(\T^2)}+
\norm{s_0(x')-s_b(0)}_{C^{3,\alpha}(\T^2)}+\norm{u_0'(x')}_{C^{3,\alpha}(\T^2;\mathbb{R}^2)}\\
+\norm{p_1(x')-p_b(L)}_{C^{3,\alpha}(\T^2)}
\le \varepsilon\le \varepsilon_0,
\end{multline}
then there is uniquely one solution $U$ to Problem (S),  with $p\in C^{3,\alpha}(\overline{\Omega})$, $s, E, u\in C^{2,\alpha}(\overline{\Omega})$, and
	\begin{eqnarray}\label{eq214}
	\norm{p-p_b}_{C^{3,\alpha}(\overline{\Omega})}+
	\norm{s-s_b}_{C^{2,\alpha}(\overline{\Omega})}+
	\norm{E-E_b}_{C^{2,\alpha}(\overline{\Omega})}+
	\norm{u'}_{C^{2,\alpha}(\overline{\Omega};\mathbb{R}^2)}
	\le C\varepsilon.
	\end{eqnarray}
\end{theorem}

\begin{remark}
The technical S-Condition is given by Definition \ref{def41} in Section \ref{sec4}. It is shown there that except
at most countable infinite numbers $\mu$, all the background solutions $U_b$ determined by the rest $\mu$ satisfy the S-Condition.
\end{remark}

\section{Reformulation of Problem (S)}\label{sec3}

The following is an important theorem established in  \cite[p.703]{LXY2016}), which is rewritten for use to our case, namely, $\Omega$ is a flat manifold. Here $D_u f=u\cdot\mathrm{grad} f$ for a vector field $u$ and a function $f$ on $\Omega$, and as a convention, repeated Roman indices will be summed up for $0,1,2$, while repeated Greek indices are to be summed over for $1,2$, except otherwise stated.
\begin{proposition}\label{prop31}
Suppose that $p \in C^2(\Omega)\cap C^1(\overline{\Omega})$, $\rho, u\in C^1(\overline{\Omega})$, and $\rho>0, u^0\ne0$ in $\overline{\Omega}$. Then  $p, \rho, u$ solve the system \eqref{eq101}-\eqref{eq103} in $\Omega$ if and only if they satisfy the following equations in $\Omega$:
\begin{eqnarray}
&&D_uE-\mathfrak{b}\cdot u=0,\label{eq36}\\
&&D_uA(s)=0,\label{eq37}\\
&&D_u\left(\frac{D_u p}{\gamma p}\right)-\di \left(\frac{\grad\,
p}{\rho}\right)- \p_ju^k\p_ku^j+\di\,\mathfrak{b}\nonumber \\ &&\qquad+L^0(\frac{D_up}{\gamma p}+\di\,u)+L^1(D_uE-\mathfrak{b}\cdot u)\nonumber\\
&&\qquad\qquad+L^2(D_uA(s))
+L^3(D_uu+\frac{\grad\,p}{\rho}-\mathfrak{b})=0,\label{eq38}\\
&&D_u u^\beta+\frac{\p_\beta\, p}{\rho}=0,\quad\beta=1,2,\label{eq39}
\end{eqnarray}
and the boundary condition
\begin{multline}
\frac{D_u p}{\gamma p}+\di\,u+L_1(D_uE-\mathfrak{b}\cdot u)+L_2(D_uA(s))
+L_3(D_uu+\frac{\grad\,p}{\rho}-\mathfrak{b})=0 \qquad \text{on}\quad \Sigma_0.\label{eq310}
\end{multline}
Here $L^0(\cdot)$ is a linear function, and  $L^k(\cdot)$, $L_k(\cdot)$ are smooth functions so that $L^k(0)=0, L_k(0)=0$ for $k=1,2,3$.
\end{proposition}

To formulate a nonlinear Problem (S1) which is equivalent to Problem (S), we need to compute the exact expressions of \eqref{eq38} and \eqref{eq310}, then specify the auxiliary functions $L^k, L_k$ appeared in Proposition \ref{prop31}.

\subsection{Problem (S1)}

\subsubsection{The equation of pressure}
We now compute the explicit expression of equation
\eqref{eq38}. It is straightforward to check that
\begin{eqnarray}\label{eq311}
&&D_u\left(\frac{D_up}{\gamma p}\right)-\di\,\left(\frac{\grad\, p}{\rho}\right)-\p_ju^k\p_ku^j
+\di\,\mathfrak{b}\nonumber\\
&=&\frac{1}{\gamma p}\left[\Big((u^0)^2-c^2\Big)\p_0^2p-c^2(\p_1^2p+\p_2^2p)\right]\nonumber\\
&&\qquad+\frac{1}{\gamma p}\left[u^0\p_0u^0\p_0p-\frac{(u^0)^2}{p}(\p_0p)^2+\frac{\gamma p}{\rho^2} \p_0\rho\p_0p\right]\nonumber\\
&&\qquad\qquad-(\p_0u^0)^2{-2\mu u^0\p_0u^0}+F_1(U),
\end{eqnarray}
where
\begin{eqnarray}\label{eq312}
F_1(U)=\sum_{(k,j)\ne(0,0)}\left(\frac{u^ku^j}{\gamma p}\p_{jk}p
+\frac{u^k}{\gamma p}\p_ku^j\p_j p-\frac{1}{\gamma p^2}u^ku^j\p_kp\p_jp-\p_ju^k\p_ku^j
+\frac{1}{\rho^2}\delta_{kj}\p_k\rho\p_j p\right),
\end{eqnarray}
and $\delta_{kj}$ is the standard Kronecker delta.
Replacing terms like $u^0\p_0u^0, (\p_0u^0)^2, (u^0)^2$
in \eqref{eq311} by using suitable equations \eqref{eq101}-\eqref{eq103} (this is why we introduced $L^k$),
after some straightforward computations, we get the identity
\begin{eqnarray}\label{eq313}
&&D_u\left(\frac{D_up}{\gamma p}\right)-\di\,\left(\frac{\grad\, p}{\rho}\right)-\p_ju^k\p_ku^j
+\di\,\mathfrak{b}\nonumber\\
&=&\frac{1}{\gamma	 p}\left[\Big(2E-\frac{\gamma+1}{\gamma-1}c^2\Big)\p_0^2p-c^2(\p_1^2p+\p_2^2p)\right]
-\frac{1}{\gamma p}2\mu(E-\frac{ c^2}{\gamma-1})\p_0p\nonumber\\
&&\qquad-\frac{2}{\gamma p^2}\left(E-\frac{c^2}{\gamma-1}+\frac{c^4}{4\gamma}\frac{1}{E-\frac{c^2}{\gamma-1}}
\right)(\p_0p)^2+2\mu^2\left(E-\frac{c^2}{\gamma-1}\right)\nonumber\\
&&\qquad\qquad+F_1+F_2,
\end{eqnarray}
with
\begin{eqnarray}\label{eq314}
F_2&=&-\left((u^1)^2+(u^2)^2\right)
	\left[\frac{1}{\gamma p}\p_0^2p
+\frac{(\p_0p)^2}{\gamma p^2}\left(-1+\frac{c^4}{\gamma}
\frac{1}{2E-\frac{2c^2}{\gamma-1}}\frac{1}{2E-
	\left((u^1)^2+(u^2)^2\right)-\frac{2c^2}{\gamma-1}}\right)\right]\nonumber\\
&&\qquad+\frac{1}{\gamma p}\left[\left(\mu\left((u^1)^2+(u^2)^2\right) -u^\beta\p_{\beta}u^0\right)\p_0p+\rho^{\gamma-1}\p_0p\frac{u^\beta}{u^0}\p_\beta A(s)\right]\nonumber\\
&&\qquad\qquad-\frac{u^\beta\p_\beta u^0}{(u^0)^2}\left(u^\beta\p_\beta u^0+2\frac{\p_0p}{\rho}\right)-\mu^2\left((u^1)^2+(u^2)^2\right).
\end{eqnarray}
Multiplying $\gamma p$ to both sides of \eqref{eq313}, and comparing it with \eqref{eq38}, we see  \eqref{eq38} is equivalent to the following second order equation of pressure
\begin{eqnarray}\label{eq315}
N(U)&\triangleq&\left[\Big(2E-\frac{\gamma+1}{\gamma-1}c^2\Big)\p_0^2p-c^2(\p_1^2p+\p_2^2p)\right]
-2\mu(E-\frac{ c^2}{\gamma-1})\p_0p\nonumber\\
&&\qquad-\frac{2}{ p}\left(E-\frac{c^2}{\gamma-1}+\frac{c^4}{4\gamma}\frac{1}{E-\frac{c^2}{\gamma-1}}\right)(\p_0p)^2
+2\mu^2\gamma p\left(E-\frac{c^2}{\gamma-1}\right)\nonumber\\
&&\qquad\qquad=F_3(U)\triangleq-\gamma p(F_1+F_2).
\end{eqnarray}

\subsubsection{The boundary conditions}
It follows, by setting  $$L_3(D_uu+\frac{\grad\,p}{\rho}-\mathfrak{b})=-\frac{1}{u^0}(D_uu^0+\frac{\p_0 p}{\rho}+\mu (u^0)^2),$$
we have the identity
\begin{eqnarray}\label{eq316}
&&\rho u^0\left(\frac{D_up}{\gamma p}+\di\, u+L_3(D_uu+\frac{\grad\,p}{\rho}-\mathfrak{b})\right)\nonumber\\
&=&\left(\frac{(u^0)^2}{c^2}-1\right)\p_0p-\mu\rho (u^0)^2
+\rho u^0\p_\beta u^\beta\nonumber\\
&&+{u^0u^\beta}\left(\frac{1}{c^2}+\frac{1}{(u^0)^2}
\right)\p_\beta p
+\frac{1}{\gamma-1}\frac{u^\beta}{u^0}\rho^{\gamma}\p_\beta A(s)
+\frac{\rho}{u^0} u^\sigma u^\beta\p_\beta u^\sigma-\frac{\rho u^\beta}{u^0}\p_\beta E.
\end{eqnarray}
Comparing this to \eqref{eq310}, it is a nonlinear
Robin condition for $p$ on $\Sigma_0$:
\begin{eqnarray}\label{eq317}
\p_0p -\frac{\mu\gamma(u^0)^2}{(u^0)^2-c^2} p
&=&G_1(U)+G_2(U),
\end{eqnarray}
with
\begin{eqnarray}\label{eq318}
G_1&\triangleq&-\frac{1}{\frac{(u^0)^2}{c^2}-1}\rho u^0\p_\beta u^\beta,\\
G_2&\triangleq&-\frac{1}{\frac{(u^0)^2}{c^2}-1}\left[{u^0u^\beta}\left(\frac{1}{c^2}+\frac{1}{(u^0)^2}
\right)\p_\beta p+\frac{1}{\gamma-1}\frac{u^\beta}{u^0}\rho^{\gamma}\p_\beta A(s)
\right.\nonumber\\
&&\qquad\left.+\frac{\rho}{u^0}u^\sigma u^\beta \p_\beta u^\sigma-\frac{\rho u^\beta}{u^0}\p_\beta E\right].\label{eq319}
\end{eqnarray}

\subsubsection{Problem (S1)}
From the above computations, by Proposition \ref{prop31}, we see that Problem (S) could be written equivalently as the following Problem (S1), for those unknowns $p,\rho, u$ with regularity as assumed in Proposition  \ref{prop31}.

\medskip
\fbox{
	\parbox{0.90\textwidth}{
		Problem (S1): {Solve  the system \eqref{eq36}, \eqref{eq37}, \eqref{eq315} and \eqref{eq39} in $\Omega$, subjected to the boundary conditions \eqref{eq21}, \eqref{eq22}, and \eqref{eq317}.}}}
\medskip

\subsection{Problem (S2)}

In this subsection we separate the linear main parts from the nonlinear equations \eqref{eq36}, \eqref{eq315} and \eqref{eq317}, and write them in the form $\mathcal{L}(U-U_b)=\mathcal{N}(U-U_b)$, where $\mathcal{L}$ is a linear operator, and $\mathcal{N}(U)$ are certain higher-order terms defined below.

\begin{definition}\label{def301}
	For $U=p, E, \rho, s, u$ etc., set $\hat{U}=U-U_b.$ A higher-order
	term is an expression containing either
	
	(i) $\hat{U}|_{\p\Omega}$ and/or its first-order tangential derivatives;
	
	\noindent
	or
	
	(ii) product of $\hat{U}$, and/or their derivatives $D\hat{U}, D^2\hat{U}$, where $D^ku$ is a $k^{th}$-order  derivative of $u$.
\end{definition}

\subsubsection{Linearization of Bernoulli law}
Recall that for the background solution, $E_b$ solves the equation $\frac{\dd }{\dd x^0}E_b=-\mu u_b^2$. So we get the identity
\begin{eqnarray}\label{eqber}
D_uE-\mathfrak{b}\cdot u&=&D_u\hat{E}+\mu u^0\Big((u^0)^2-(u_b)^2\Big)\nonumber\\
&=&D_u\hat{E}+2\mu u^0\Big(\hat{E}-\frac{1}{\gamma-1}(c^2-c_b^2)\Big)-\mu u^0\left((u^1)^2+(u^2)^2\right).
\end{eqnarray}
Using expressions like
\begin{eqnarray}
c^2-c_b^2
&=&\frac{\gamma-1}{\rho_b}\hat{p}+\rho_b^{\gamma-1}\widehat{A(s)}+O(1)
(|\hat{p}|^2+|\widehat{A(s)}|^2),
\label{eqc}
\end{eqnarray}
where $O(1)$ represents a bounded quantity with a bound depending only on the background solution ${U_b}$ and $|U-{U_b}|$, after straightforward calculations,  \eqref{eq36} is equivalent to
\begin{eqnarray}\label{eqhatber}
D'_u \hat{E}+2\mu \hat{E}&\triangleq&\p_0 \hat{E}+\frac{u^1}{u^0}\p_1 \hat{E}+\frac{u^2}{u^0}\p_2 \hat{E}+2\mu \hat{E}\nonumber\\
&=&\frac{2\mu }{\gamma-1}\rho_{b}^{\gamma-1}\widehat{A(s)}+\frac{2\mu}{\rho_b} \hat{p}+H,
\end{eqnarray}
with
\begin{eqnarray}\label{eqH}
H=\mu\left((u^1)^2+(u^2)^2\right)+O(1)\frac{2\mu}{\gamma-1}(|\hat{p}|^2+|\widehat{A(s)}|^2).
\end{eqnarray}

\subsubsection{Linearization of pressure's equation}
For \eqref{eq315}, note that $N(U_b)=0$,  we may get a linear operator $\mathcal{L}$ and write $\mathcal{L}(\hat{U})-F_4(U)=N(U)-N(U_b)$, with $F_4(U)$ a higher-order term. Then  \eqref{eq315} becomes $\mathcal{L}(\hat{U})=F_3(U)+F_4(U)$.
In fact, using \eqref{eqc}, by setting $t={u_b^2}/{c_b^2}=M_b^2\in (0,1),$
direct computation yields that \eqref{eq315} can be written as
\begin{eqnarray}\label{eq324}
\mathcal{L}(\hat{p})&\triangleq&(t-1)\p_0^2\hat{p}-\p_1^2\hat{p}-\p_2^2\hat{p}+\mu
d_1(t)\p_0\hat{p}+\mu^2d_2(t)\hat{p}+\mu^2\rho_bd_3(t)\hat{E}
+\mu^2\rho_b^\gamma d_4(t)\widehat{A(s)}\nonumber\\
&=&F_5\triangleq \frac{1}{c_b^2}(F_3+F_4).
\end{eqnarray}

We easily see that \eqref{eq324} is an elliptic equation of (perturbed) pressure. The coefficients in \eqref{eq324} are given by
\begin{eqnarray}
d_1(t)&\triangleq&
-\frac{1}{t-1}\left((1+2\gamma)t^2+t-2\right),\label{eq325}\\
d_2(t)&\triangleq&\frac{1}{(t-1)^3}\left(\gamma(1+\gamma)t^4-
2\gamma(1+\gamma)t^3-(\gamma-3)t^2-2(4+\gamma)t+8\right), \label{eq326}\\
d_3(t)&\triangleq&\frac{1}{(t-1)^3}\Big(\gamma t^2+3t-4\Big),\label{eq327}\\
d_4(t)&\triangleq&-\frac{1}{\gamma-1}\frac{1}{(t-1)^3}\left(\gamma(\gamma-1)t^3+(5\gamma-3)
t^2-2(\gamma-4)t-8\right),\label{eq328}
\end{eqnarray}
and
\begin{eqnarray}\label{eq329}
-F_4
&=&\left(2\hat{E}-\frac{\gamma+1}{\gamma-1}(c^2-c_b^2)\right)\p_0^2\hat{p}
-(c^2-c_b^2)\p_\beta^2\hat{p}
+2\mu^2\gamma\hat{p}\left(\hat{E}-\frac{1}{\gamma-1}
(c^2-c_b^2)\right)\nonumber\\
&&\quad-2\mu\p_0\hat{p}\left(\hat{E}-\frac{1}{\gamma-1}(c^2-c_b^2)\right)
-\p_0(p+p_b)\p_0\hat{p}\left[\frac{|u|^2}{p}-\frac{u_b^2}{p_b}\right]
-\frac{u_b^2}{p_b}(\p_0\hat{p})^2\nonumber\\
&&\quad\quad+(\p_0p_b)^2\frac{\hat{p}}{p_b}\left[\frac{|u|^2}{p}-\frac{u_b^2}{p_b}\right]
-\frac{\p_0(p+p_b)\p_0\hat{p}}{\gamma}\left(\frac{c^4}{p|u|^2}-\frac{c_b^4}{p_bu_b^2}\right)
-\frac{c_b^4}{\gamma p_bu_b^2}(\p_0\hat{p})^2\nonumber\\
&&\quad\quad\quad-\frac{(\p_0p_b)^2}{\gamma}\left(\frac{1}{|u|^2}-\frac{1}{u_b^2}\right)
\left[\left(\frac{c^4}{p}-\frac{c_b^4}{p_b}\right)+\frac{2c_b^4}{p_bu_b^2}
\left(\frac{c^2-c_b^2}{\gamma-1}-\hat{E}\right)\right]\nonumber\\
&&\quad\quad\quad\quad-\frac{(\p_0p_b)^2}{\gamma u_b^2}\left[\left(\frac{c^2+c_b^2}{p}-\frac{2c_b^2}{p_b}\right)(c^2-c_b^2)
-\frac{c_b^4}{p_b}\hat{p}\left(\frac{1}{p}-\frac{1}{p_b}\right)\right]
\nonumber\\
&&\quad\quad\quad\quad\quad\quad-\left[\frac{\gamma+1}{\gamma-1}\p_0^2p_b
+\frac{2\mu^2\gamma}{\gamma-1}p_b
-\frac{2\mu}{\gamma-1}\p_0p_b
-\frac{2}{\gamma-1}\frac{(\p_0p_b)^2}{p_b}\right.\nonumber\\
&&\quad\quad\quad\quad\quad\quad\quad\left.+\frac{2c_b^2}{\gamma p_b u_b^2}(\p_0p_b)^2\left(1+\frac{1}{\gamma-1}\frac{c_b^2}{u_b^2}\right)\right]
\times\left[O(1)(|\hat{p}|^2+|\widehat{A(s)}|^2)\right].
\end{eqnarray}

\subsubsection{Linearization of boundary condition}
Note that ${p_b}$ solves \eqref{eq27}, so \eqref{eq317} is equivalent to
\begin{eqnarray}\label{eq330}
\p_0(p-p_b)-\mu\gamma\left(\frac{ p
	(u^0)^2}{(u^0)^2-c^2}-\frac{p_b
	(u_b)^2}{(u_b)^2-c_b^2}\right)&=&G_1+G_2.
\end{eqnarray}
Using expressions in \eqref{eqc},  \eqref{eq330} could be written as
\begin{eqnarray}\label{eq331}
\p_0\hat{p}+\gamma_0\hat{p}=G\triangleq G_1+G_2+G_3,
\end{eqnarray}
with $\gamma_0$ a negative constant determined by the background solution at $x^0=0$:
\begin{eqnarray}\label{eq332}
\gamma_0\triangleq-\mu\frac{\gamma M_b(0)^4-M_b(0)^2+2}{(M_b(0)^2-1)^2}<0,
\end{eqnarray}
and
\begin{eqnarray}\label{eq333}
G_3&=&\mu\gamma\left\{\left(\frac{(u^0)^2}{(u^0)^2-c^2}-\frac{u_b^2}{u_b^2-c_b^2}\right)\hat{p}
+p_b\frac{u_b^2(c^2-c_b^2)-c_b^2((u^0)^2-u_b^2)}{u_b^2-c_b^2}
\left[\frac{1}{(u^0)^2-c^2}-\frac{1}{u_b^2-c_b^2}\right]\right.\nonumber\\
&&\qquad\left.-\frac{p_bc_b^2}{(u_b^2-c_b^2)^2}[2\underline{(E-E_b)-
(u^1)^2-(u^2)^2}]+\frac{2p_bE_b}{(u_b^2-c_b^2)^2}\left[O(1)(|\hat{p}|^2+|\widehat{A(s)}|^2)
\right.\right.\nonumber\\
&&\qquad\qquad\left.\left.+\rho_b^{\gamma-1}\underline{({A(s)-A(s_b)})}
\right]\right\}.
\end{eqnarray}
We note that $G_2, G_3$ are higher-order terms (the terms with underlines are given by boundary data, so fulfill the item $(\rmnum{1})$ in Definition \ref{def301}), while $G_1$ depends on the boundary values of $u'$ on $\Sigma_0$.

\subsubsection{Problem (S2)}
We now reformulate Problem (S1) equivalently as the following Problem (S2).

\medskip
\fbox{
	\parbox{0.90\textwidth}{
		Problem (S2): Solve functions $U=(\hat{E}, \widehat{A(s)}, \hat{p}, u')$ that satisfying the following problems \eqref{eq334}--\eqref{eq337}.}}

\begin{eqnarray}
&&\begin{cases}\label{eq334}
 D'_u \hat{E}+2\mu \hat{E}=\frac{2\mu }{\gamma-1}\rho_{b}^{\gamma-1}\widehat{A(s)}+\frac{2\mu}{\rho_b} \hat{p}+H&\text{in}\quad \Omega,\\
\hat{E}=E_0(x')-E_b(0)&\text{on}\quad \Sigma_0;
\end{cases}\\
&&\begin{cases}\label{eq335}
D_u\widehat{A(s)}=0&\text{in}\quad \Omega,\\
\widehat{A(s)}=A(s_0(x'))-A(s_b(0))&\text{on}\quad \Sigma_0;
\end{cases}\\
&&\begin{cases}\label{eq336}
\mathcal{L}(\hat{p})=F_5& \text{in}\quad \Omega,\\
\p_0\hat{p}+\gamma_0\hat{p}=G &\text{on}\quad \Sigma_0,\\
\hat{p}=p_1(x')-p_b(L) &\text{on}\quad \Sigma_1;
\end{cases}\\
&&\begin{cases}\label{eq337}
D_u u^\beta+\frac{\p_\beta\, p}{\rho}=0& \text{in}\quad \Omega,\ \ \beta=1,2,\\
u^\beta=u^\beta_0(x') &\text{on}\quad \Sigma_0.
\end{cases}
\end{eqnarray}
Note that  \eqref{eq336} is a mixed boundary value problem of a second order elliptic equation, while \eqref{eq334}, \eqref{eq335} and \eqref{eq337} are Cauchy problems of transport equations.

\subsection{Problem (S3)}
Since the elliptic problem \eqref{eq336} is coupled with the other hyperbolic problems, we need to further reformulate Problem (S2) equivalently as the following Problem (S3).

We firstly consider the Cauchy problem \eqref{eq334}:
\begin{eqnarray}
\begin{cases}\label{eq41}
D'_u \hat{E}+2\mu \hat{E}=\frac{2\mu }{\gamma-1}\rho_{b}^{\gamma-1}\widehat{A(s)}+\frac{2\mu}{\rho_b} \hat{p}+H&\text{in}\quad \Omega,\\
\hat{E}=E_0(x')-E_b(0)&\text{on}\quad \Sigma_0.
\end{cases}
\end{eqnarray}
For the vector field $u$ defined in $\Omega$, we consider the non-autonomous vector field $\frac{u'}{u^0}(x^0, x')$ defined for $x'=(x^1, x^2)\in \T^2$ and $x^0\in[0, L]$. For $\bar{x}\in \T^2$, we write the integral curve passing $(0, \bar{x})$
as $x'=\varphi(x^0, \bar{x})$, which is a $C^{k,\alpha}$ function in $\Omega$ if $u\in C^{k,\alpha}(\bar{\Omega})$ and $u^0>\delta$ for a positive  constant $\delta$, and $k\in\mathbb{N}$. For fixed $x^0$, the map $\varphi_{x^0}: \T^2\to \T^2, \ \bar{x}\mapsto x'=\varphi(x^0, \bar{x})$  is a $C^{k,\alpha}$ homeomorphism.

\begin{lemma}\label{lem41}
Suppose that $u=(u^0,u')\in C^{0,1}$ and $u^0>\delta$. There is a positive constant $C=C(\delta,L)$ so that for any $x'\in \T^2$ and $x_0\in[0, L]$, it holds
\begin{eqnarray}\label{eq42}
\left|(\varphi_{x^0})^{-1}x'-x'\right|\le C\norm{u'}_{C^0(\bar{\Omega})}.
\end{eqnarray}
\end{lemma}

\begin{proof}
Let $\bar{x}=(\varphi_{x^0})^{-1}x'$. Then
$\left|\varphi_{x^0}(\bar{x})-\bar{x}\right|\le\int_{0}^{x^0}\left|\frac{u'}{u^0}(\tau, \varphi(\tau,\bar{x}))\right|\,\dd \tau\le C\norm{u'}_{C^0(\bar{\Omega})}$
as desired.
\end{proof}

We write the unique solution to the linear transport equation \eqref{eq41} as follows:
\begin{eqnarray}\label{eq43}
\hat{E}(x)&=&\hat{E}(x^0, \varphi_{x^0}(\bar{x}))=e^{-2\mu x^0}\hat{E}_0(\bar{x})\nonumber\\
&&+\int_{0}^{x^0}e^{2\mu(\tau-x^0)}    \left(\frac{2\mu}{\gamma-1}\rho_{b}^{\gamma-1}\widehat{A(s)}+\frac{2\mu}{\rho_b} \hat{p}+H\right)(\tau, \varphi_\tau(\bar{x}))\,\dd \tau.
\end{eqnarray}
Set
\begin{eqnarray}\label{eqF6}
F_6=-\mu^2\rho_bd_3(t)\int_{0}^{x^0}e^{2\mu(\tau-x^0)}    \left(\frac{2\mu}{\rho_b} \hat{p}(\tau, \varphi_\tau(\bar{x}))-\frac{2\mu}{\rho_b} \hat{p}(\tau, x')+H(\tau, \varphi_\tau(\bar{x}))\right)\,\dd \tau,
\end{eqnarray}
which is a higher order term (note that $\p_\beta \hat{p}$ itself is small, and $\varphi_{x^0}$ is close to the identity map since $u'$ is nearly zero, so $|(\varphi_{x^0})^{-1}(x')-x'|$ is small by \eqref{eq42}). Then
we could write the elliptic equation in \eqref{eq336}  as
\begin{eqnarray}\label{eq45}
&&\Big(t(x^0)-1\Big)\p_0^2\hat{p}-\p_1^2\hat{p}-\p_2^2\hat{p}+\mu d_1(t)\p_0\hat{p}+\mu^2d_2(t)\hat{p}
+2\mu^3e^{-2\mu x^0}\rho_bd_3(t)
\int_{0}^{x^0}\frac{ e^{2\mu\tau}}{\rho_b(\tau)}\hat{p}(\tau, {x}')\,\dd \tau
\nonumber\\
&=&-\mu^2\rho_b^\gamma d_4(t)\widehat{A(s)}-\mu^2 e^{-2\mu x^0}\rho_bd_3(t)    \left(\hat{E}_0(\bar{x})
+\int_{0}^{x^0}\frac{2\mu}{\gamma-1}e^{2\mu\tau}\rho_{b}^{\gamma-1}    \widehat{A(s)}(\tau, \varphi_\tau(\bar{x}))\,\dd \tau\right)\nonumber\\
&&\qquad+F_5+F_6.
\end{eqnarray}

Now set
\begin{equation*}
\begin{array}{lllrll}
e_1(x^0)& = &(t(x^0)-1)<0,&e_2(x^0)&=&{\mu}d_1(t(x^0)),\\
e_3(x^0)& = &\mu^2d_2(t(x^0)),&
e_4(x^0)&=&2\mu^3e^{-2\mu x^0}\rho_b(x^0)d_3(t(x^0)),\\
e_5(x^0)& = &-\mu^2(\rho_b(x^0))^{\gamma}d_4(t(x^0)),&
e_6(x^0)&=&-\mu^2e^{-2\mu x^0}\rho_b(x^0)d_3(t(x^0)),\\
b(\mu,\tau)&=&\frac{ e^{2\mu \tau}}{\rho_b(\tau)},& F& =& F_5+{F}_6,\\
\end{array}
\end{equation*}
and
\begin{eqnarray}
F_0&=&e_5(x^0)\widehat{A(s)}+e_6(x^0)\left(\hat{E}_0(\bar{x})+\int_{0}^{x^0}\frac{2\mu}{\gamma-1}e^{2\mu\tau}\rho_{b}^{\gamma-1}\widehat{A(s)}(\tau, \varphi_\tau(\bar{x}))\,\dd \tau\right).\label{eqF0}
\end{eqnarray}
Then equation \eqref{eq45} simply reads
\begin{eqnarray}\label{eq47}
\mathfrak{L}(\hat{p})&\triangleq& e_1(x^0)\p_0^2\hat{p}-\p_1^2\hat{p}-\p_2^2\hat{p}+e_2(x^0)\p_0\hat{p}+e_3(x^0)\hat{p}
+e_4(x^0)\int_{0}^{x^0}b(\mu,\tau) \hat{p}(\tau, {x}')\,\dd \tau\nonumber \\
&=&F_0+F.
\end{eqnarray}
Note that there is a nonlocal term $e_4(x^0)\int_{0}^{x^0}{b}(\mu,\tau)\hat{p}(\tau, {x}')\,\dd \tau$. This is quite different from \cite{LXY2016} and shows the spectacular influence of friction. So problem \eqref{eq336} can be rewritten as follows:
\begin{eqnarray}
\begin{cases}\label{eq48}
\mathfrak{L}(\hat{p})=F_0+F& \text{in}\quad \Omega,\\
\p_0\hat{p}+\gamma_0\hat{p}=G &\text{on}\quad \Sigma_0,\\
\hat{p}=p_1(x')-p_b(L) &\text{on}\quad \Sigma_1.
\end{cases}
\end{eqnarray}

We then state Problem (S3), which is equivalent to Problem (S2), as can be seen from the above derivations.

\smallskip
\fbox{
\parbox{0.90\textwidth}{
Problem (S3): Find functions $U=(\hat{E}, \widehat{A(s)}, \hat{p}, u')$ defined in $\Omega$ that solve problems  \eqref{eq334}, \eqref{eq335}, \eqref{eq48} and \eqref{eq337}.}}

\section{A linear second order nonlocal elliptic equation with mixed boundary conditions}\label{sec4}

To attack problem \eqref{eq48},  we study in this section the linear second order nonlocal elliptic equation subjected to Robin condition on $\Sigma_0$ and Dirichlet condition on $\Sigma_1$:
\begin{eqnarray}\label{eq49}
\begin{cases}
\mathfrak{L}(\hat{p})=h(x)& \text{in}\quad \Omega,\\
\p_0\hat{p}+\gamma_0\hat{p}=g_0(x') &\text{on}\quad \Sigma_0,\\
\hat{p}=g_1(x') &\text{on}\quad \Sigma_1.
\end{cases}
\end{eqnarray}
Here $h\in C^{k-2,\alpha}(\bar{\Omega}), g_0\in C^{k-1,\alpha}(\T^2)$ and $g_1\in C^{k,\alpha}(\T^2)$ are given nonhomogeneous terms, and $k=2,3,\ldots$

The interesting part is that the elliptic operator $\mathfrak{L}$ contains a nonlocal term. So we could not use directly maximum principles or energy estimates, and we need the S-Condition to avoid some possible spectrums.

\subsection{Uniqueness of solutions}
We firstly study under what conditions a strong solution $\hat{p}$ to problem \eqref{eq49} in Sobolev space $H^2(\Omega)$ is unique. To this end, we consider the homogeneous problem
\begin{eqnarray}\label{eq410}
\begin{cases}
\mathfrak{L}(\hat{p})=0& \text{in}\quad \Omega,\\
\p_0\hat{p}+\gamma_0\hat{p}=0 &\text{on}\quad \Sigma_0,\\
\hat{p}=0 &\text{on}\quad \Sigma_1.
\end{cases}
\end{eqnarray}

By Sobolev Embedding Theorem, $H^2(\Omega)$ is embedded into $C^{0,\frac12}(\bar{\Omega})$, so the solution is bounded.  Now considering
$e_4(x^0)\int_{0}^{x^0}{b}(\mu,\tau) \hat{p}(\tau, {x}')\,\dd \tau$
as  a nonhomogeneous term,   the standard theory of second order elliptic equations with Dirichlet and oblique derivative conditions \cite{GT} imply that for any $\alpha\in (0,\frac12)$, $\hat{p}$ belongs to  $C^{2,\alpha}(\bar{\Omega})$. Then using a standard boot-strap argument, we deduce that if $\hat{p}\in H^2(\Omega)$ is a strong solution to \eqref{eq410}, then $\hat{p}$ is a classical solution and  belongs to $C^\infty(\bar{\Omega})$.

Now we wish to show that $\hat{p}\equiv0$. The idea is to use the method of separation of variables via the Fourier series.

Denote $m=(m_1,m_2)$, then by the above consideration of regularity, we could  write
\begin{eqnarray}\label{eq411}
\hat{p}(x)&=&\sum_{m_1,m_2=0}^{\infty}\lambda_{m}\left\{ p_{1,m}(x^0)\cos(m_1x^1)\cos(m_2x^2)+p_{2,m}(x^0)\sin(m_1x^1)\cos(m_2x^2)\right.\nonumber\\
&&\qquad\left.+p_{3,m}(x^0)\cos(m_1x^1)\sin(m_2x^2)+p_{4,m}(x^0)\sin(m_1x^1)\sin(m_2x^2)\right\},
\end{eqnarray}
with the coefficients
\begin{eqnarray*}
\lambda_{m}=
\begin{cases}
\frac{1}{4}& \text{if}\quad m_1=m_2=0,\\
\frac{1}{2} &\text{if only one of}\ \ m_1,\, m_2\ \ \text{is}\ \ 0,\\
1 &\text{if}\quad m_1>0,m_2>0.
\end{cases}
\end{eqnarray*}

Recall that $x=(x^0, x^1,x^2)$,  then the coefficients in \eqref{eq411} are given by
\begin{eqnarray*}
p_{1,m}(x^0)&=&\frac{1}{\pi^2}\int_{\T^2}\hat{p}(x)\cos(m_1x^1)\cos(m_2x^2)\,\dd x^1\dd x^2,\\
p_{2,m}(x^0)&=&\frac{1}{\pi^2}\int_{\T^2}\hat{p}(x)\sin(m_1x^1)\cos(m_2x^2)\,\dd x^1\dd x^2,\\
p_{3,m}(x^0)&=&\frac{1}{\pi^2}\int_{\T^2}\hat{p}(x)\cos(m_1x^1)\sin(m_2x^2)\,\dd x^1\dd x^2,\\
p_{4,m}(x^0)&=&\frac{1}{\pi^2}\int_{\T^2}\hat{p}(x)\sin(m_1x^1)\sin(m_2x^2)\,\dd x^1\dd x^2.\\
\end{eqnarray*}
For $\hat{p}\in C^{k,\alpha}(\bar{\Omega})$ and $k\ge 2$, we easily deduce that  $p_{i,m}(x^0)$ $(i=1,2,3,4)$ belongs to $C^{k,\alpha}([0, L])$ for all $k$, and it is also true that the series \eqref{eq411} converges uniformly in $C^{k-2}(\bar{\Omega})$.

Substituting \eqref{eq411} into \eqref{eq410}, for $x^0\in[0, L]$,  each $p_{i,m}(x^0)$ solves the following nonlocal ordinary differential equation:
\begin{eqnarray}\label{eq412}
e_1(x^0)p_{i,m}''+e_2(x^0)p_{i,m}'+(e_3(x^0)
+|m|^2)p_{i,m}+e_4(x^0)\int_{0}^{x^0}{b}(\mu,\tau) p_{i,m}(\tau)\,\dd \tau=0,
\end{eqnarray}
and the two-point boundary conditions:
\begin{eqnarray}\label{eq413}
p_{i,m}'(0)+\gamma_0p_{i,m}(0)=0 ,\qquad p_{i,m}(L)=0.
\end{eqnarray}
We need to find conditions to guarantee that all $p_{i,m}\,  ( i=1,2,3,4)$ are zero.

Suppose that $p_{i,m}(0)=0,$  we set $\mathcal{P}_{i,m}(x^0)=\int_{0}^{x^0}{b}(\mu,\tau) p_{i,m}(\tau)\,\dd \tau$. Then the system \eqref{eq412} and \eqref{eq413} can be rewritten as
\begin{eqnarray}\label{eq-pp}
\begin{cases}
\tilde{e}_1(x^0)\mathcal{P}_{i,m}'''
+\tilde{e}_2(x^0)\mathcal{P}''_{i,m}+\tilde{e}_3(x^0)\mathcal{P}'_{i,m}+e_4(x^0)\mathcal{P}_{i,m}=0,\\
\mathcal{P}_{i,m}(0)=\mathcal{P}_{i,m}'(0)=\mathcal{P}_{i,m}''(0)=0,\\
\mathcal{P}_{i,m}'(L)=0.
\end{cases}
\end{eqnarray}
Here we define
\begin{align*}
\tilde{e}_1(x^0)=&\frac{e_1(x^0)}{{b}(\mu,x^0)}<0,\\
\tilde{e}_2(x^0)=&\left(\frac{e_2(x^0)}{{b}(\mu,x^0)}-\frac{2e_1(x^0){b}'(\mu,x^0)}{{b}^2(\mu,x^0)}
\right),\\
\tilde{e}_3(x^0)=&\left(\frac{(e_3(x^0)+|m|^2)}{{b}(\mu,x^0)}-\frac{e_2(x^0){b}'(\mu,x^0)
+e_1(x^0){b}''(\mu,x^0)}{{b}^2(\mu,x^0)}
+\frac{2e_1(x^0)({b}'(\mu,x^0))^2}{{b}^3(\mu,x^0)}
\right),
\end{align*}
and $b'(\mu,x_0)=\p b(\mu,x_0)/\p x^0, b''(\mu,x_0)=\p^2 b(\mu,x_0)/{(\p x^0)}^2$.
By uniqueness of solutions of Cauchy problem of ordinary differential equations, obviously one has that $p_{i,m}\equiv 0$ in $[0, L]$.

If $p_{i,m}(0)\ne0$ for some $i, m$, we set $w_{i,m}(x^0)=\frac{p_{i,m}(x^0)}{p_{i,m}(0)}$ and $\mathcal{W}_{i,m}(x^0)=\int_{0}^{x^0}{b}(\mu,\tau) w_{i,m}(\tau)\,\dd \tau$. Then it solves
\begin{eqnarray}\label{eq415}
\begin{cases}
\tilde{e}_1(x^0)\mathcal{W}_{i,m}'''
+\tilde{e}_2(x^0)\mathcal{W}''_{i,m}+\tilde{e}_3(x^0)\mathcal{W}'_{i,m}+e_4(x^0)\mathcal{W}_{i,m}=0,\\
\mathcal{W}_{i,m}(0)=0,\quad \mathcal{W}_{i,m}'(0)=b(\mu,0),
\quad \mathcal{W}_{i,m}''(0)=b'(\mu,0)-\gamma_0b(\mu,0),\\
\mathcal{W}_{i,m}'(L)=0.
\end{cases}
\end{eqnarray}

\begin{definition}\label{def41}
We say a background solution $U_b$ satisfies the {\it S-Condition}, if for each $i=1,2,3,4$ and $m\in\Z^2$,  problem \eqref{eq415} does {\it not} have a classical solution.
\end{definition}

If the background solution $U_b$ satisfies the S-Condition, then all $p_{i,m}$ are zero, hence problem \eqref{eq410} has only the trivial solution. Our purpose below is to show theoretically that almost all background solutions satisfy the S-Condition. Actually, we have the following lemma.

\begin{lemma}\label{lem42}
There exists a set $\mathcal{S}\subset(0, +\infty)$ of at most countable infinite points such that the background solutions $U_b$ determined by $\mu\in(0, +\infty)\setminus \mathcal{S}$ satisfy the S-Condition.
\end{lemma}

\begin{proof}
We note that the background solution $U_b$, and all the coefficients $e_1, e_2, e_3, e_4$ and $b$  depend analytically on the parameter $\mu$. Hence the unique solution $\mathcal{W}_{i,m}$ to this Cauchy problem \eqref{eq415} is also real analytic with respect to the parameter $\mu$ (cf. \cite{walter}). We write it as $\mathcal{W}_{i,m}=\mathcal{W}_{i,m}(x^0; \mu)$. Particulary, $\vartheta_{i,m}(\mu)\triangleq \mathcal{W}_{i,m}'(L; \mu)=b(\mu,L)w_{i,m}(L)$ is real analytic for $\mu\geq 0$.

For given $i=1,2,3,4, m\in\Z^2,\mu_0>0$, suppose now there are infinite numbers of  $\mu\in[0,\mu_0]$ so that $\vartheta_{i,m}(\mu)=0$. Then by compactness of $[0, \mu_0]$, the function $\vartheta_{i,m}$ has a non-isolated zero point. So it must be identically zero and we have $\vartheta_{i,m}(0)=0.$ However, for $\mu=0$, problem \eqref{eq415} is reduced to
\begin{eqnarray}\label{eqmu0}
\begin{cases}
e_1(x^0)w_{i,m}''+|m|^2w_{i,m}=0,\\
w_{i,m}(0)=1,\qquad w_{i,m}'(0)=0.
\end{cases}
\end{eqnarray}

We note that $e_1(x^0)<0$. Hence maximal principle and Hopf boundary point lemma (cf. \cite{GT}) imply that $w_{i,m}(L;0)\neq 0$, namely $\vartheta_{i,m}(0)\neq 0$, contradicts to our conclusion that $\vartheta_{i,m}(0)=0$. So for each fixed $m\in\Z^2,\mu_0>0$, there are at most finite numbers of zeros of $\vartheta_{i,m}$. Therefore, there are at most countable infinite numbers of $\mu$ so that the problem \eqref{eq415} may have a solution. The conclusion of the lemma then follows.
\end{proof}

\subsection{A priori estimate}
By considering the nonlocal term
\[
e_4(x^0)\int_{0}^{x^0}{b}(\mu,\tau) \hat{p}(\tau, {x}')\,\dd \tau
\]
as part of the non-homogenous term,  and applying the standard theory of second order elliptic equations in \cite{GT} for the Dirichlet and oblique derivative problem, with the aid of a standard higher regularity argument, and interpolation inequalities in \cite{GT},  we infer that  any $\hat{p}\in C^{k,\alpha}(\bar{\Omega})$ ($k=2,3$) solves problem \eqref{eq49} should satisfy the estimate
\begin{eqnarray}\label{eq417}
\norm{\hat{p}}_{C^{k,\alpha}(\bar{\Omega})}\le
C\Big(\norm{\hat{p}}_{C^0(\bar{\Omega})}+\norm{h}_{C^{k-2,\alpha}(\bar{\Omega})}
+\norm{g_0}_{C^{k-1,\alpha}(\T^2)}+\norm{g_1}_{C^{k,\alpha}(\T^2)}\Big),
\end{eqnarray}
with  $C$ a constant depending only on the background solution $U_b$ and $\alpha$.

Then by \eqref{eq417}
and a compactness argument as in \cite[p.738]{LXY2016}, we have the {\it a priori} estimate:
\begin{eqnarray}\label{eq418}
\norm{\hat{p}}_{C^{k,\alpha}(\bar{\Omega})}\le
C\Big(\norm{h}_{C^{k-2,\alpha}(\bar{\Omega})}
+\norm{g_0}_{C^{k-1,\alpha}(\T^2)}+\norm{g_1}_{C^{k,\alpha}(\T^2)}\Big)
\end{eqnarray}
for any $C^{k,\alpha}$ solution of problem \eqref{eq49}, provided that the only solution to problem \eqref{eq410} is zero.

\subsection{Approximate solutions}
We now use Fourier series to establish a family of approximate solutions to problem \eqref{eq49}.

Without loss of generality, we take $g_1=0$  in the sequel. 
We also set $\{h^{(n)}\}_n$ to be a sequence of $C^\infty(\bar{\Omega})$ functions that converges to $h$ in $C^{k-2,\alpha}(\bar{\Omega})$, and $\{g^{(n)}_0\}_n\subset C^\infty(\T^2)$ converges to $g_0$ in $C^{k-1,\alpha}(\T^2)$. Now for fixed $n$, we consider problem \eqref{eq49}, with $h$ there replaced by $h^{(n)}$, and $g_0$ replaced by $g^{(n)}_0$.

Suppose that
\begin{eqnarray}
h^{(n)}(x)&=&\sum_{m_1,m_2=0}^{\infty}\lambda_{m}\left\{ h^{(n)}_{1,m}(x^0)\cos(m_1x^1)\cos(m_2x^2)+h^{(n)}_{2,m}(x^0)\sin(m_1x^1)
\cos(m_2x^2)\right.\nonumber\\
&&\left.+h^{(n)}_{3,m}(x^0)\cos(m_1x^1)\sin(m_2x^2)+h^{(n)}_{4,m}(x^0)\sin(m_1x^1)
\sin(m_2x^2)\right\},\label{eq420}\\
g^{(n)}_0(x')&=&\sum_{m_1,m_2=0}^{\infty}\lambda_{m}\left\{ (g_0^{(n)})_{1,m}\cos(m_1x^1)\cos(m_2x^2)+(g_0^{(n)})_{2,m}\sin(m_1x^1)
\cos(m_2x^2)\right.\nonumber\\
&&\left.+(g_0^{(n)})_{3,m}\cos(m_1x^1)\sin(m_2x^2)+(g_0^{(n)})_{4,m}\sin(m_1x^1)
\sin(m_2x^2)\right\}.\label{eq421}
\end{eqnarray}
Then for $\hat{p}$ given by \eqref{eq411}, each $\mathcal{P}_{i,m}(x^0), i=1,2,3,4$  should solve the following two-point boundary value problem of a third order ordinary differential equation:
\begin{eqnarray}\label{eq422}
\begin{cases}
\tilde{e}_1(x^0)\mathcal{P}_{i,m}'''
+\tilde{e}_2(x^0)\mathcal{P}''_{i,m}+\tilde{e}_3(x^0)\mathcal{P}'_{i,m}+e_4(x^0)\mathcal{P}_{i,m}={h}^{(n)}_{i,m}(x^0)\hspace{5em} x^0\in[0, L],\\
\mathcal{P}_{i,m}(0)=0,
\quad \mathcal{P}_{i,m}''(0)
+\left(\gamma_0-\frac{b'(\mu,0)}{b(\mu,0)}\right)\mathcal{P}_{i,m}'(0)=b(\mu,0)(g_0^{(n)})_{i,m},\\
\mathcal{P}_{i,m}'(L)=0.
\end{cases}
\end{eqnarray}

Actually, we know that, under the S-Condition, the homogeneous system has only the trivial solution. So by standard theory of linear ordinary differential equations, there is one and only one solution to problem \eqref{eq422}. Note that ${h}^{(n)}_{i,m}(x^0)\in C^{\infty}([0, L])$ as  $h^{(n)}\in C^{\infty}(\bar{\Omega})$, and the coefficients in \eqref{eq422} are all real analytic, so the solution $p_{i,m}(x^0)=\frac{1}{b(\mu,x^0)}\mathcal{P}_{i,m}'(x^0)$ belongs to  $C^{\infty}([0, L]).$

Now for $N\in\mathbb{N}$, we define
\begin{eqnarray*}
\hat{p}_N(x)&=&\sum_{m_1,m_2=0}^{N}\lambda_{m}\left\{ p_{1,m}(x^0)\cos(m_1x^1)\cos(m_2x^2)+p_{2,m}(x^0)\sin(m_1x^1)\cos(m_2x^2)\right.\\
&&\left.+p_{3,m}(x^0)\cos(m_1x^1)\sin(m_2x^2)+p_{4,m}(x^0)\sin(m_1x^1)\sin(m_2x^2)\right\},\\
h^{(n)}_N(x)&=&\sum_{m_1,m_2=0}^{N}\lambda_{m}\left\{ h^{(n)}_{1,m}(x^0)\cos(m_1x^1)\cos(m_2x^2)+h^{(n)}_{2,m}(x^0)\sin(m_1x^1)
\cos(m_2x^2)\right.\\
&&\left.+h^{(n)}_{3,m}(x^0)\cos(m_1x^1)\sin(m_2x^2)+h^{(n)}_{4,m}(x^0)\sin(m_1x^1)
\sin(m_2x^2)\right\},\\
(g^{(n)}_0)_N(x')&=&\sum_{m_1,m_2=0}^{N}\lambda_{m}\left\{ (g_0^{(n)})_{1,m}\cos(m_1x^1)\cos(m_2x^2)+(g_0^{(n)})_{2,m}\sin(m_1x^1)
\cos(m_2x^2)\right.\\
&&\left.+(g_0^{(n)})_{3,m}\cos(m_1x^1)\sin(m_2x^2)+(g_0^{(n)})_{4,m}\sin(m_1x^1)
\sin(m_2x^2)\right\}.
\end{eqnarray*}
Apparently  $\hat{p}_{N}, \ h^{(n)}_N\in C^{\infty}(\bar{\Omega})$, and $(g_0^{(n)})_{N}\in C^\infty(\T^2)$. It is also easy to check that $\hat{p}_{N}$ solves the following problem:
\begin{eqnarray}\label{eq424}
\begin{cases}
\mathfrak{L}(\hat{p}_N)=h^{(n)}_N& \text{in}\quad \Omega,\\
\p_0\hat{p}_N+\gamma_0\hat{p}_N=(g_0^{(n)})_N &\text{on}\quad \Sigma_0,\\
\hat{p}_N=0 &\text{on}\quad \Sigma_1.
\end{cases}
\end{eqnarray}

\subsection{Existence}
By the estimate \eqref{eq418}, for any $N_1,N_2\in\mathbb{N}$ with $N_1<N_2$, there holds
\begin{eqnarray*}
\norm{\hat{p}_{N_2}-\hat{p}_{N_1}}_{C^{k,\alpha}(\bar{\Omega})}&\le& C\Big(\norm{h^{(n)}_{N_2}-h^{(n)}_{N_1}}_{C^{k-2,\alpha}(\bar{\Omega})}
+\norm{(g_0^{(n)})_{N_2}-(g_0^{(n)})_{N_1}}_{C^{k-1,\alpha}(\T^2)}\Big).
\end{eqnarray*}
Recall that $h^{(n)}_N\to h^{(n)}$ in $C^{k-2,\alpha}(\bar{\Omega})$ and $(g_0^{(n)})_N\to g_0^{(n)}$ in $C^{k-1,\alpha}(\T^2)$, we infer that  $\{\hat{p}_N\}$  is a Cauchy sequence in $C^{k,\alpha}(\bar{\Omega})$. So there is a $\hat{p}^{(n)}\in C^{k,\alpha}(\bar{\Omega})$  and $\hat{p}_N\to \hat{p}^{(n)}$ in $C^{k,\alpha}(\bar{\Omega})$ as $N\to \infty$. Taking the limit $N\to\infty$ in problem  \eqref{eq424}, one sees that $\hat{p}^{(n)}$ is a $C^{k,\alpha}(\bar{\Omega})$ solution to problem \eqref{eq49}, with $h$ there replaced by $h^{(n)}$, $g_0$ replaced by $g_0^{(n)}$.

Now for the approximate solutions $\{\hat{p}^{(n)}\}_n$, we use the estimate \eqref{eq418} to infer that
\begin{eqnarray*}
\norm{\hat{p}^{(n)}}_{C^{k,\alpha}(\bar{\Omega})}&\le& C\Big(\norm{h^{(n)}}_{C^{k-2,\alpha}(\bar{\Omega})}
+\norm{g_0^{(n)}}_{C^{k-1,\alpha}(\T^2)}\Big)\nonumber\\
&\le&C\Big(\norm{h}_{C^{k-2,\alpha}(\bar{\Omega})}
+\norm{g_0}_{C^{k-1,\alpha}(\T^2)}\Big).
\end{eqnarray*}
Hence by Ascoli--Arzela Lemma, there is a subsequence of $\{\hat{p}^{(n)}\}$ that converges to some $\hat{p}\in C^{k,\alpha}(\bar{\Omega})$ in the norm of $C^{k}(\bar{\Omega})$. Taking limit with respect to this subsequence in the boundary value problems of $\hat{p}^{(n)}$, we easily see that $\hat{p}$ is a classical solution to problem \eqref{eq49}. Therefore, we proved the following proposition.

\begin{proposition}\label{prop41}
Suppose that the S-Condition holds. Then problem \eqref{eq49} has one and only one solution in $C^{k,\alpha}(\bar{\Omega})$, and it satisfies the estimate \eqref{eq418}.
\end{proposition}

\section{Stability of subsonic Fanno flows}\label{sec5}

For $k=2,3$, and $U=(p,s,E, u'=(u^1,u^2))$, suppose that $p\in C^{k,\alpha}(\overline{\Omega})$ and $s,E, u'\in C^{k-1,\alpha}(\overline{\Omega})$. We define the  norm
\begin{eqnarray}\label{eq425}
\norm{U}_k&\triangleq& \norm{p}_{C^{k,\alpha}(\overline{\Omega})}+
\norm{s}_{C^{k-1,\alpha}(\overline{\Omega})}+
\norm{E}_{C^{k-1,\alpha}(\overline{\Omega})}+
\sum_{\beta=1}^2\norm{u^\beta}_{C^{k-1,\alpha}(\overline{\Omega})}.
\end{eqnarray}
Let $K$ be a positive number to be chosen. We define $X_{K\varepsilon}$ to be the (non-empty) set of functions $U$ so that \eqref{eq21} and \eqref{eq22} hold, and
$\norm{U-U_b}_3\le K\varepsilon.$
To prove Theorem \ref{thm21}, we construct a mapping $\mathcal{T}$ on $X_{K\varepsilon}$ for suitably chosen $K$ and $\varepsilon$, and show that it contracts under the norm $\norm{\cdot}_2$. Then by a Banach fixed point theorem, $\mathcal{T}$ has uniquely one fixed point $U\in X_{K\varepsilon}$, which is exactly a solution to Problem (S3).

\subsection{Construction of mapping $\mathcal{T}$}
For any $U\in X_{K\varepsilon}$, by the following process we define a mapping $\mathcal{T}: U\mapsto \tilde{U}$.  Set $\hat{U}=\tilde{U}-U_b$. Then we only need to determine $\hat{U}$. We also use $C$ to denote generic positive constants which might be different in different places.
\medskip

\subsubsection{Determination of $\widetilde{A(s)}$}

Noting the equation \eqref{eq335}, we solve the unknowns $\widehat{A(s)}$ from the following Cauchy problem of linear transport equation, where the velocity field $u$ is given as part of $U\in X_{K\varepsilon}$:
\begin{eqnarray}
\begin{cases}\label{eq426}
D_u\widehat{A(s)}=0&\text{in}\quad \Omega,\\
\widehat{A(s)}=A(s_0(x'))-A(s_b(0))&\text{on}\quad \Sigma_0.
\end{cases}
\end{eqnarray}
Once we solved $\widehat{A(s)}$, we get $ \widetilde{A(s)}=A(s_b)+\widehat{A(s)}.$
Since $u\in C^{2,\alpha}(\overline{\Omega})$, and recall that $s_0(x')-s_b(0)\in C^{3,\alpha}(\T^2)$, by Lemma A.1 in \cite[p.751]{LXY2016},  we have

\begin{lemma}\label{lem43}
	There is uniquely one $\widehat{A(s)}\in C^{2,\alpha}(\overline{\Omega})$ solves \eqref{eq426}. In addition,
	\begin{eqnarray}\label{eq427}
\norm{\widehat{A(s)}}_{C^{2,\alpha}
		(\overline{\Omega})}&\le& C\norm{A(s_0)-A(s_b)}_{C^{2,\alpha}
		({\T^2})}\le C\varepsilon,
	\end{eqnarray}
The second inequality holds provided that \eqref{eq213} is valid.
\end{lemma}

\subsubsection{ Determination of $\tilde{p}$}
Now consider the following mixed boundary value problem of $\hat{p}$ (comparing to problem \eqref{eq48}):
\begin{eqnarray}\label{eq428}
\begin{cases}
e_1(x^0)\p_0^2\hat{p}-\p_1^2\hat{p}-\p_2^2\hat{p}+e_2(x^0)\p_0\hat{p}+e_3(x^0)\hat{p}\\
\hspace{4em}+e_4(x^0)\int_{0}^{x^0}{b}(\mu,\tau) \hat{p}(\tau, {x}')\,\dd \tau
=F_0+F&\text{in}\quad \Omega,\\
\p_0\hat{p}+\gamma_0\hat{p}=G &\text{on}\quad \Sigma_0,\\
\hat{p}=p_1(x')-p_b(L) &\text{on}\quad \Sigma_1.
\end{cases}
\end{eqnarray}
Note here that for the term $\widehat{A(s)}$ in nonhomogeneous  term $F_0$ (see \eqref{eqF0}), we take $\widehat{A(s)}$ to be the functions solved from Lemma \ref{lem43}. Then we obtain that $\norm{F_0(U)}_{C^{1,\alpha}(\overline{\Omega})}\le C\varepsilon.$

We now specify the nonhomogeneous  term $F(U)$. In $F_6$ and $F_4$ (see \eqref{eqF6} and \eqref{eq329}), all $\hat{U}$ should be replaced by $U-U_b$ (recall that the $U\in X_{K\varepsilon}$ has been fixed), by direct computations we get that
$\norm{F_6}_{C^{1,\alpha}(\overline{\Omega})}+\norm{F_4}_{C^{1,\alpha}(\overline{\Omega})}\le CK^2\varepsilon^2.$
In the expression of $F_3$ (see \eqref{eq315}), we take all $U$ to be the given one.  So by the smallness of tangential velocity $u'$, we have
$\norm{\gamma p(F_1+F_2)}_{C^{1,\alpha}(\overline{\Omega})}\le CK^2\varepsilon^2,$ and  $\norm{F_3}_{C^{1,\alpha}(\overline{\Omega})}\le CK^2\varepsilon^2.$
Therefore we obtain that
\begin{eqnarray}\label{eq429}
\norm{F(U)}_{C^{1,\alpha}(\overline{\Omega})}\le C(\varepsilon+K^2\varepsilon^2).
\end{eqnarray}

Next we specify the term $G(U)$ in boundary conditions.
For $G_1$ (see \eqref{eq318}), the tangential velocity $u'$ should be  the given boundary data $u'_{0}(x')$ (hence belong to $C^{3,\alpha}(\T^2; \R^2)$), and the others are replaced by the given $U$. So by \eqref{eq214}, we have
$\norm{G_1}_{C^{2,\alpha}(\T^2)}\le C\varepsilon.$
For the term $G_2$ (see \eqref{eq319}), $u'$ and $A(s), E$ should be the given boundary data $u'_{0}(x'), A(s_0)(x'), E_0(x')$, respectively, while the others are replaced by the given $U\in X_{K\varepsilon}$. Hence it still lies in $C^{2,\alpha}(\T^2)$ and we have
$\norm{G_2}_{C^{2,\alpha}(\T^2)}\le CK\varepsilon^2.$
For the expression of $G_3$ (see \eqref{eq333}), except the
underline terms are replaced by the given boundary data, all the other $U$ can be taken as the given $U$ in $X_{K\varepsilon}$, and it easily follows that
$\norm{G_3}_{C^{2,\alpha}(\T^2)}\le C(\varepsilon+K^2\varepsilon^2).$
So in all, we obtain
\begin{eqnarray}\label{eq430}
\norm{G(U)}_{C^{2,\alpha}(\T^2)}\le C(\varepsilon+K^2\varepsilon^2).
\end{eqnarray}

Then, since we assume that the S-Condition holds, by Proposition \ref{prop41}, we have the following lemma.
\begin{lemma}\label{lem44}
There is uniquely one solution $\hat{p}\in C^{3,\alpha}(\overline{\Omega})$ to problem  \eqref{eq428}. Moreover, there holds that
	\begin{eqnarray}\label{eq431}
	\norm{\hat{p}}_{C^{3,\alpha}(\overline{\Omega})}&\le& C\left(\norm{F_0(U)}_{C^{1,\alpha}(\overline{\Omega})}
+\norm{F(U)}_{C^{1,\alpha}(\overline{\Omega})}
+\norm{G(U)}_{C^{2,\alpha}(\T^2)}+\norm{p_1-p_b}_{C^{3,\alpha}(\T^2)}\right)
\nonumber\\	
	&\le& C(\varepsilon+K^2\varepsilon^2).
	\end{eqnarray}
\end{lemma}
Hence we obtain that $\tilde{p}=\hat{p}+p_b$.

\subsubsection{ Determination of $\tilde{E}$}

By the equations \eqref{eq334}, we solve the unknowns $\hat{E}$ from the following Cauchy problems of linear transport equations, where the velocity field $u$ is given as part of $U\in X_{K\varepsilon}$:
\begin{eqnarray}
\begin{cases}\label{eq432}
D'_u \hat{E}+2\mu \hat{E}=\frac{2\mu }{\gamma-1}\rho_{b}^{\gamma-1}\widehat{A(s)}+\frac{2\mu}{\rho_b} \hat{p}+H&\text{in}\quad \Omega,\\
\hat{E}=E_0(x')-E_b(0)&\text{on}\quad \Sigma_0.
\end{cases}
\end{eqnarray}

Once we solved $\hat{E}$, we get $\tilde{E}=E_b+\hat{E}.$ Note here that for the two terms $\widehat{A(s)}$ and $\hat{p}$ on the right-hand side, we take $\widehat{A(s)}$  and $\hat{p}$ to be the functions solved from Lemma \ref{lem43} and Lemma \ref{lem44}, respectively. For the nonhomogeneous term $H(U)$ (see \eqref{eqH}), all $\hat{U}$ should be replaced by $U-U_b$ (recall that the $U\in X_{K\varepsilon}$ has been fixed). So by direct computations,  we get that
$\norm{H}_{C^{2,\alpha}(\overline{\Omega})}\le CK^2\varepsilon^2.$ Since $u\in C^{2,\alpha}(\overline{\Omega})$, and recall that ${E_0(x')-E_b(0)}\in C^{3,\alpha}(\T^2)$,   we have

\begin{lemma}\label{lem45}
	There is a unique $\hat{E}\in C^{2,\alpha}(\overline{\Omega})$ that solves \eqref{eq432}. In addition,
	\begin{eqnarray}\label{eq433}
	\norm{\hat{E}}_{C^{2,\alpha}(\overline{\Omega})}&\le& C\left(\norm{E_0-E_b}_{C^{2,\alpha}(\T^2)}+\norm{\widehat{A(s)}}_{C^{2,\alpha}
		(\overline{\Omega})}+\norm{\hat{p}}_{C^{2,\alpha}
		(\overline{\Omega})}+\norm{H}_{C^{2,\alpha}
		(\overline{\Omega})}\right)\nonumber\\
&\le& C(\varepsilon+K^2\varepsilon^2).
	\end{eqnarray}
Here the estimates \eqref{eq213}, \eqref{eq427} and \eqref{eq431} are used to obtain the second inequality.
\end{lemma}

\subsubsection{Determination of $\tilde{u}'$}
Now we solve $\tilde{u}^\beta$ ($\beta=1,2$) from the transport equations \eqref{eq337}. For $\beta=1$, we have the problem
\begin{eqnarray}\label{eq434}
\begin{cases}
u^j\p_j\tilde{u}^1+\frac{1}{\rho}\p_1\tilde{p}=0 &\text{in}\quad\Omega,\\
\tilde{u}^1=u^1_0(x') &\text{on}\quad \Sigma_0.
\end{cases}
\end{eqnarray}
Here $u^1_0(x')$ is the given boundary data, and $\tilde{p}$ is given by Lemma \ref{lem44}. Then we have the following lemma.
\begin{lemma}\label{lem46}
	There is uniquely one solution $\tilde{u}^1\in C^{2,\alpha}(\overline{\Omega})$ to problem \eqref{eq434}. In addition,
	\begin{eqnarray}\label{eq435}
	\norm{\tilde{u}^1}_{C^{2,\alpha}(\overline{\Omega})}\le C\Big( \norm{u^1_0}_{C^{2,\alpha}(\T^2)}+\norm{\hat{p}}_{C^{3,\alpha}(\overline{\Omega})}\Big)
\le C\Big(\varepsilon+K^2\varepsilon^2\Big).
	\end{eqnarray}
The second inequality holds provided that \eqref{eq213} is valid.
\end{lemma}

We have similar results for $\tilde{u}^2$.

\subsubsection{Conclusion}
So for any fixed $U\in X_{K\varepsilon}$ we obtain uniquely one  $\tilde{U}=(\tilde{p}, \tilde{s}, \tilde{E}, \tilde{u}')$, and the estimates \eqref{eq427} \eqref{eq431} \eqref{eq433} and \eqref{eq435} yield that
\[
\norm{\tilde{U}-U_b}_3\le C(\varepsilon+K^2\varepsilon^2).
\]
By choosing $K=\max\{2C, 1\}$, and $\varepsilon_0\le \min\{1/K^2,1\}$, we  have
\begin{eqnarray}\label{equ}
\norm{\tilde{U}-U_b}_3\le K\varepsilon
\end{eqnarray}
for all  $\varepsilon\le \varepsilon_0.$
Hence we proved that $\tilde{U}\in X_{K\varepsilon}$, and the mapping $\mathcal{T}: X_{K\varepsilon}\to X_{K\varepsilon}$ is well-defined.

\subsection{Contraction of the mapping $\mathcal{T}$}
For $U^{(1)}, U^{(2)}\in X_{K\varepsilon}$, set $\tilde{U}^{(i)}=\mathcal{T}(U^{(i)})$, $i=1,2$. We will  show below that the mapping
\[
\mathcal{T}: X_{K\varepsilon}\to X_{K\varepsilon}, \quad U\mapsto \tilde{U}
\]
is a contraction in the sense that
\begin{eqnarray}\label{eq437}
\norm{\tilde{U}^{(1)}-\tilde{U}^{(2)}}_2\le\frac12\norm{U^{(1)}-U^{(2)}}_2,
\end{eqnarray}
provided that $\varepsilon_0$ is further small. To prove \eqref{eq437}, we consider the problems satisfied by $\bar{U}=\tilde{U}^{(1)}-\tilde{U}^{(2)}$.

\noindent {\it Step 1.} We note that $\overline{A(s)}$ solves the following problem ({\it cf.} \eqref{eq426})
\begin{eqnarray*}
\begin{cases}
D_{u^{(1)}}\overline{A(s)}+D_{u^{(1)}-u^{(2)}}\widehat{A(s)}^{(2)}=0&\text{in}\quad \Omega,\\
\overline{A(s)}=0&\text{on}\quad \Sigma_0;
\end{cases}
\end{eqnarray*}

By  Lemma \ref{lem43}, we have
\begin{eqnarray}\label{eq438}
\norm{\overline{A(s)}}_{C^{1,\alpha}(\bar{\Omega})}\le C \norm{u^{(1)}-u^{(2)}}_{C^{1,\alpha}(\bar{\Omega})}
\norm{\widehat{A(s)}^{(2)}}_{C^{2,\alpha}(\bar{\Omega})}
\le C\varepsilon \norm{U^{(1)}-U^{(2)}}_2.
\end{eqnarray}

\noindent {\it Step 2.} Next we seek an estimate of $\bar{p}$, which solves  ({\it cf.} \eqref{eq428})
\begin{eqnarray*}
\begin{cases}
\mathfrak{L}(\bar{p})=\bar{F}_0+F(U^{(1)})-F(U^{(2)}) & \text{in}\quad \Omega,\\
\p_0\bar{p}+\gamma_0\bar{p}=G(U^{(1)})-G(U^{(2)}) &\text{on}\quad \Sigma_0,\\
\bar{p}=0 &\text{on}\quad \Sigma_1,
\end{cases}
\end{eqnarray*}
with
\[
\bar{F}_0=e_5(x^0)\overline{A(s)}+e_6(x^0)\int_{0}^{x^0}\frac{2\mu}{\gamma-1}
e^{2\mu\tau}\rho_{b}^{\gamma-1}\overline{A(s)}(\tau, \varphi_\tau(\bar{x}))\,\dd \tau.
\]

By Lemma \ref{lem44} and  \eqref{eq438},  direct computation shows
\begin{eqnarray}\label{eq439}
\norm{\bar{p}}_{C^{2,\alpha}(\overline{\Omega})}&\le& C\left(\norm{\overline{A(s)}}_{C^{0,\alpha}(\overline{\Omega})}
+\norm{F(U^{(1)})-F(U^{(2)})}_{C^{0,\alpha}(\overline{\Omega})}
+\norm{G(U^{(1)})-G(U^{(2)})}_{C^{1,\alpha}(\T^2)}\right)
\nonumber\\	
	&\le& C\varepsilon \norm{U^{(1)}-U^{(2)}}_2.
\end{eqnarray}

\noindent{\it Step 3.} Note that $\bar{E}$ solves the following problem ({\it cf.} \eqref{eq432})
\begin{eqnarray*}
\begin{cases}
D'_{u^{(1)}} \bar{E}+2\mu \bar{E}+D_{u^{(1)}-u^{(2)}}\hat{E}^{(2)}=\frac{2\mu }{\gamma-1}\rho_{b}^{\gamma-1}\overline{A(s)}+\frac{2\mu}{\rho_b} \bar{p}+H(U^{(1)})-H(U^{(2)})&\text{in}\quad \Omega,\\
\bar{E}=0&\text{on}\quad \Sigma_0.
\end{cases}
\end{eqnarray*}
Then using Lemma \ref{lem45}, \eqref{eq438} and \eqref{eq439}, one has
\begin{eqnarray}\label{eq440}
\norm{\bar{E}}_{C^{1,\alpha}(\bar{\Omega})}&\le&C\left(
\norm{\overline{A(s)}}_{C^{1,\alpha}(\bar{\Omega})}
+\norm{\bar{p}}_{C^{1,\alpha}(\bar{\Omega})}
+\norm{u^{(1)}-u^{(2)}}_{C^{1,\alpha}(\bar{\Omega})}
\norm{\hat{E}^{(2)}}_{C^{2,\alpha}(\bar{\Omega})}\right.\nonumber\\
&&\qquad\left. +\norm{H(U^{(1)})-H(U^{(2)})}_{C^{1,\alpha}(\bar{\Omega})}\right)\nonumber\\
&\le&C\varepsilon \norm{U^{(1)}-U^{(2)}}_2.
\end{eqnarray}

\noindent {\it Step 4.} From \eqref{eq434}, $\bar{u}^1$ solves
\begin{eqnarray*}
\begin{cases}
(u^j)^{(1)}\p_j\bar{u}^1+\Big((u^j)^{(1)}
-(u^j)^{(2)}\Big)\p_j(\tilde{u}^1)^{(2)}
=-\frac{1}{\rho^{(1)}}\p_1 \bar{p}+\left(\frac{1}{\rho^{(2)}}-\frac{1}{\rho^{(1)}}\right)\p_1 \tilde{p}^{(2)} &\text{in}\quad \Omega,\\
\bar{u}^1=0&\text{on}\quad \Sigma_0.
\end{cases}
\end{eqnarray*}
So by Lemma \ref{lem46}, and  \eqref{eq439}, there holds
\begin{eqnarray}\label{441}
\norm{\bar{u}^1}_{C^{1,\alpha}(\bar{\Omega})}&\le& C\left(\norm{\bar{p}}_{C^{2,\alpha}(\bar{\Omega})}
+\norm{u^{(1)}-u^{(2)}}_{C^{1,\alpha}(\bar{\Omega})}
\norm{(\tilde{u}^1)^{(2)}}_{C^{2,\alpha}(\bar{\Omega})}\right.\nonumber\\
&&\qquad\left.+\norm{\rho^{(1)}-\rho^{(2)}}_{C^{1,\alpha}(\bar{\Omega})}
\norm{\hat{p}^{(2)}}_{C^{2,\alpha}(\bar{\Omega})}\right)\nonumber\\
&\le&C\varepsilon \norm{U^{(1)}-U^{(2)}}_2.
\end{eqnarray}
There is a similar estimate  for the second component $\bar{u}^2$.

\medskip \noindent {\it Conclusion.} Now summing up the inequalities
\eqref{eq438}--\eqref{441}, we get
\begin{eqnarray*}
\norm{\tilde{U}^{(1)}-\tilde{U}^{(2)}}_2\le C\varepsilon\norm{U^{(1)}-U^{(2)}}_2,
\end{eqnarray*}
which implies \eqref{eq437} if $\varepsilon\in(0, \varepsilon_0)$ and  $C\varepsilon_0<1/2$.

Therefore, by Banach fixed point theorem, $\mathcal{T}$ has one and only one fixed point,  say $U$, in $X_{K\varepsilon}$. By the construction of the mapping $\mathcal{T}$, the fixed point is a solution to Problem (S3). On the contrary, for a solution to Problem (S3) which lies in $X_{K\varepsilon}$, it must be a fixed point of $\mathcal{T}$.  The proof of  Theorem \ref{thm21} is completed.

\bigskip
{\bf Acknowledgments.}
Hairong Yuan is supported in part by National Nature Science Foundation of China under
Grant No.11371141, and by Science and Technology Commission of Shanghai Municipality (STCSM) under grant No. 13dz2260400.



\begin{thebibliography}{99}
\bibitem{BDX}
M. Bae; B. Duan; C. Xie. Subsonic flow for the multidimensional Euler-Poisson system. Arch. Ration. Mech. Anal. 220 (2016), no. 1, 155--191.

\bibitem{xiechen}
C. Chen; C. Xie. Three dimensional steady subsonic Euler flows in bounded nozzles. J. Differential Equations  256  (2014),  no. 11, 3684--3708.

\bibitem{cdx}
G.-Q. Chen; X. Deng; W. Xiang. Global steady subsonic flows through infinitely long nozzles for the full Euler equations. SIAM J. Math. Anal.  44  (2012),  no. 4, 2888--2919.


\bibitem{CY}
S. Chen; H. Yuan. Transonic shocks in compressible flow passing a duct for three-dimensional Euler systems. Arch. Ration. Mech. Anal.  187  (2008),  no. 3, 523--556.

\bibitem{CHHQ2016}
S.-W. Chou; J. M. Hong; B.-C. Huang; R. Quita. Global transonic solutions to combined Fanno Rayleigh flows through variable nozzles. arXiv:1611.10083, 2016.

\bibitem{CF}
R. Courant; K. O. Friedrichs. Supersonic flow and shock waves. Applied Mathematical Sciences, Vol. 21. Springer-Verlag, New York-Heidelberg, 1976.


\bibitem{Da}
C.~M. Dafermos. { Hyperbolic Conservation Laws in Continuum Physics}. Third edition. Grundlehren der mathematischen Wissenschaften, vol. 325. Springer-Verlag, Berlin Heidelberg, 2010.


\bibitem{Far2008}
S. Farokhi. Aircraft Propulsion. John Wiley \& Sons, Hoboken, NJ, 2008.


\bibitem{GT}
D. Gilbarg; N. S. Trudinger. Elliptic partial differential equations of second order. Reprint of the 1998 edition. Classics in Mathematics. Springer-Verlag, Berlin, 2001.

\bibitem{HPW}
F. Huang; R. Pan; Z. Wang.
$L^1$ convergence to the Barenblatt solution for compressible Euler equations with damping. Arch. Ration. Mech. Anal. 200 (2011), no. 2, 665--689.

\bibitem{LXY2016}
L. Liu; G. Xu; H. Yuan. Stability of spherically symmetric subsonic flows and
transonic shocks under multidimensional perturbations. Adv. Math. 291 (2016), 696--757.


\bibitem{Sha1953}
A. H. Shapiro. The dynamics and thermodynamics of compressible fluid flow, Vol. 1, Ronald Press Co., New York, 1953.

\bibitem{Ts2015}
N. Tsuge. Existence of global solutions for isentropic gas flow in a divergent nozzle with friction. J. Math. Anal. Appl. 426 (2015), no. 2, 971--977.

\bibitem{walter}
W. Walter. Ordinary differential equations.  Graduate Texts in Mathematics, 182. Readings in Mathematics. Springer-Verlag, New York, 1998.

\bibitem{weng}
S. Weng. A new formulation for the 3-D Euler equations with an
application to subsonic flows in a cylinder. Indiana Univ. Math. J. 64  (2015), no. 6, 1609--1642.




\bibitem{Yu1}
H. Yuan. Examples of steady subsonic flows in a convergent-divergent approximate nozzle. J. Differential Equations  244  (2008),  no. 7, 1675--1691.

\end{thebibliography}
\end{document}